\documentclass[a4paper, 11pt]{amsart}

\usepackage[utf8]{inputenc}
\usepackage[all,cmtip]{xy}
\usepackage[english]{babel}
\usepackage{mathtools}
\usepackage{graphicx}
\usepackage{multicol}

\usepackage{amsmath,amssymb,amsthm,mathrsfs}
\usepackage[alphabetic]{amsrefs}

\usepackage{xfrac}
\usepackage{pifont}

\usepackage{array,tabularx}
\usepackage{xcolor}

\usepackage{tikz,tikz-cd}
\usetikzlibrary{positioning, shapes.geometric}
\usetikzlibrary{arrows,arrows.meta}
\usetikzlibrary{calc}
\tikzcdset{arrow style=tikz, diagrams={>={Computer Modern Rightarrow[width=5.5pt,length=3pt]}}}

\usepackage{enumitem}
\usepackage{cleveref}

\usepackage[retainorgcmds]{IEEEtrantools}

\theoremstyle{plain}
\newtheorem{thm}{Theorem}[section] 
\newtheorem*{thm*}{Theorem}
\newtheorem*{mainthm}{Main Theorem}

\newtheorem{lem}[thm]{Lemma}

\counterwithin{equation}{section}

\theoremstyle{definition}
\newtheorem{defn}[thm]{Definition}
\newtheorem*{defn*}{Definition}
\newtheorem{exa}[thm]{Example}
\newtheorem{rem}[thm]{Remark}
\newtheorem*{rem*}{Remark}

\newcommand*{\myproofname}{Proof of Theorem \ref{1thm:inertness}}

\newcommand*{\myproofnames}{Proof of Proposition \ref{prop:pdhmlgy}}

\newcommand{\N}{\varmathbb{N}}

\newcommand{\C}{\varmathbb{C}}
\renewcommand{\H}{\varmathbb{H}}
\renewcommand{\O}{\varmathbb{O}}

\usepackage{geometry}
\usepackage{libertine}
\usepackage[libertine]{newtxmath}
\renewcommand{\mathbb}{\varmathbb}

\input xypic 
\input xy 
\xyoption{all} 
\usepackage{amssymb} 
\usepackage{bbm}

\newcounter{bean}

\title{Gyration Stability for Products}

\author{Sebastian Chenery}
\address{University of Bristol, School of Mathematics, Fry Building, Woodland Road, Bristol, BS8 1UG}
\email{seb.chenery@bristol.ac.uk}


\begin{document}

\begin{abstract}
    A gyration is an operation on Poincar\'e Duality complexes that arises from a certain surgery on the product of a given complex \(N\) and a sphere, parametrised by a chosen twisting. Of particular recent interest is the notion of gyration stability; that is, \(N\) is gyration stable when all of its gyrations have the same homotopy type, regardless of the twisting used. We prove that a product \(N\times M\) of two Poincar\'e Duality complexes is gyration stable when one of the product terms is itself gyration stable, and provide some examples of interest.
\end{abstract}

\maketitle


\section{Introduction} 

Gyrations first arose in geometric topology as a surgery on the Cartesian product of a given manifold and a sphere. Originally defined by Gonz{\'a}lez Acu{\~n}a \cite{gonzalezacuna}, they have since been used in many aspects of geometry and topology \cites{bosio_meersseman, klpt, huang_theriault_stability, ChenTher:pullbacks}. In particular, they occur naturally in the classification circle actions on smooth manifolds \cites{duan, galaz-garcia--reiser} and in the topology of polyhedral products, where they feature in a foundational conjecture of Gitler--L{\'o}pez-de-Medrano \cite{gitler-ldm}, which has recently been solved \cite{cfw}. There has also been much recent study of their homotopy theoretic properties by Basu-Ghosh \cite{basu-ghosh}, Huang \cite{huang_inertness24}, Huang-Theriault \cites{huangtheriault, HuangTher_stabilization}, Stanton-Theriault \cite{StanTher_skeleton-coH}, as well as by Theriault and this author in \cite{ChenTher:gy_stab} and \cite{chenery:fico}.

Throughout this article we will work with Poincar\'e Duality complexes, that is, finite \(CW\)-complexes which obey Poincar\'e Duality. Let $N$ be a path-connected \(n\)-dimensional Poincar\'{e} Duality complex with a single \(n\)-dimensional cell; examples include smooth, closed, oriented, simply-connected \(n\)-manifolds. We write \(\overline{N}\) for its \((n-1)\)-skeleton, and there is a homotopy cofibration 
\[
    S^{n-1} \xrightarrow{f_N} \overline{N} \xrightarrow{\iota_N} N
\] 
where \(f_N\) is the attaching map for the top-cell and \(\iota_N\) is the inclusion of the skeleton. Let \(k\geq 2\) be an integer and take a class \(\tau\in\pi_{k-1}(\mathrm{SO}(n))\). Using the standard linear action of \(\mathrm{SO}(n)\) on \(S^{n-1}\), define the map 
\[
    t:S^{n-1}\times S^{k-1}\rightarrow S^{n-1}\times S^{k-1}
\] 
by \(t(a, x)=(\tau(x)\cdot a,x)\). The \textit{\(k\)-gyration of \(N\) by \(\tau\)} is given by the pushout
\begin{equation*}
    \begin{tikzcd}[row sep=3em, column sep=3em]
        S^{n-1}\times S^{k-1} \arrow[r, "1\times \iota"] \arrow[d, "(f_N\times 1)\circ t"] & S^{n-1}\times D^k \arrow[d] \\
        \overline{N}\times S^{k-1} \arrow[r] & \mathcal{G}^k_\tau(N) 
    \end{tikzcd}
\end{equation*} 
where \(\iota\) is the inclusion of the boundary of the disc. When \(\tau\) is trivial the above pushout is equivalent to a \((k-1,n)\)-type surgery on \(N\times S^{k-1}\) with respect to this trivial choice. Otherwise the surgery is twisted by the action of \(\tau\), and for this reason the homotopy class \(\tau\) is referred to as a \textit{twisting} in the context of gyrations. The more geometric original formulation of Gonz{\'a}lez Acu{\~n}a is written in this notation as
\[
    \mathcal{G}_\tau^k(N)=\left((N\backslash\mathrm{Int}(D^n))\times S^{k-1}\right)\cup_t\left(S^{n-1}\times D^k\right)
\]
where \(D^n\subset N\) is an embedded \(n\)-disc centred at a chosen base point of \(N\). Of particular recent interest is the question of gyration stability, where one asks for a given \(N\) and \(k\geq 2\) whether the homotopy type of \(\mathcal{G}^k_\tau(N)\) is independent of the chosen twisting: \(N\) is called \textit{\(\mathcal{G}^k\)-stable} if \(\mathcal{G}^{k}_{\tau}(N)\simeq\mathcal{G}^{k}_{\omega}(N)\) for all twistings \(\tau,\omega\in\pi_{k-1}(\mathrm{SO}(n))\). 

For a fixed \(k\), gyration stability is equivalent to all \(k\)-gyrations having a single homotopy type, whatever twisting \(\tau\) is taken. Indeed, one can take an enumerative approach to this question: if the number of distinct homotopy types can be shown to be strictly greater than one, then gyration \textit{instability} follows. This was the method used by Theriault and the author in \cite{ChenTher:gy_stab} to show that the projective planes \(\C P^2\), \(\H P^2\) and \(\O P^2\) exhibit gyration instability for several values of \(k\). Alternatively, one can pursue a more direct computation of the homotopy type of gyrations, as was done in \cite{chenery:fico}, where it was shown that a binary sphere product \(S^p\times S^q\) (with \(q \geq p \geq 2\) with \(q\geq3\)) is \(\mathcal{G}^k\)-stable for all \(k<q\) by observing that there is a homotopy equivalence 
\[
    \mathcal{G}_\tau^k(S^p \times S^q) \simeq (S^p \times S^{q+k-1})\#(S^q \times S^{p+k-1})
\]
which is independent of \(\tau\). Spheres themselves are \(\mathcal{G}^k\)-stable for all \(k\) (cf. \cite{ChenTher:gy_stab}*{Example 4.3}) so a natural line of enquiry follows from the above example: given two Poincar\'e Duality complexes \(N\) and \(M\), if \(N\) is \(\mathcal{G}^k\)-stable, what can we say for the product \(N\times M\)?

The answer is known in the \(k=2\) case when \(N\) and \(M\) are both closed, oriented, smooth manifolds, due to Duan. For \(N\) of dimension \(n>2\) which is \(\mathcal{G}^2\)-stable, it is proved in \cite{duan} that \(N\times M\) is also \(\mathcal{G}^2\)-stable, without any additional hypotheses on \(M\). It is our purpose in this article to produce a general answer in our homotopy theoretic framework for gyrations of Poincar\'e Duality complexes. The Main Theorem is the following. 

\begin{mainthm}[Theorem \ref{thm:main}]
    Let \(n,m\) and \(k\) be integers with \(n,m>2\), \(k>1\) and \(n>k\), and suppose that \(N\) and \(M\) are two path connected Poincar\'e Duality complexes of dimension \(n\) and \(m\), respectively. Then, if \(N\) is \(\mathcal{G}^k\)-stable, so is \(N \times M\).
\end{mainthm}

We achieve this result through manipulations of the self-map \(t:S^{n+m-1}\times S^{k-1} \rightarrow S^{n+m-1}\times S^{k-1}\) that arises from the \(\mathrm{SO}(n+m)\)-action in the \(k<n\) case, and combining this with knowledge of the \(CW\)-structure of the product \(N\times M\). To do so we first recall the necessary facts about the \(CW\)-complex structure of a product in Section \ref{sec:products}. We follow this with a discussion of gyrations in Section \ref{sec:gy}, and then develop an understanding of \(\pi_{k-1}(\mathrm{SO}(n+m))\) for \(k<n\) in Section \ref{sec:action}. We then conclude with the proof of the Main Theorem in Section \ref{sec:mainthm} and give some examples of interest.

\subsubsection*{Acknowledgements} The author was supported by the Heilbronn Institute for Mathematical Research during preparation of this work.

\section{Attaching Maps for Products}\label{sec:products}

We first give a construction for the attaching map of the top-cell of a product of two Poincar\'e Duality complexes. Throughout, we assume that our Poincar\'e Duality complexes \(N\) and \(M\), of dimensions \(n\) and \(m\) respectively, are both path connected and with the structure of a $CW$-complex with a single cell in their top dimension. Denote the attaching maps of their top-cells by \(f_N\) and \(f_M\). Then there are homotopy cofibrations 
\[
    S^{n-1}\xrightarrow{f_N}\overline{N}\xrightarrow{\iota_N}N \text{\; and \;} S^{m-1}\xrightarrow{f_M}\overline{M}\xrightarrow{\iota_M}M.
\]
The \(CW\)-complexes \(\overline{N}\) and \(\overline{M}\) have the homotopy types of \(N \backslash\mathrm{Int}(D^n)\) and \(M \backslash \mathrm{Int}(D^n)\), respectively. Fix such homotopy equivalences. Under these, using \(N\) to illustrate, the map including the boundary \(S^{n-1}\) of the deleted disk in \(\overline{N}\) becomes the attaching map \(f_{N}\), and the composite \(\alpha_N:=\iota_N\circ f_N\) has the homotopy class of the restriction of an embedding \(\beta_N:D^n\rightarrow N\) to its boundary, reflecting that \(\alpha_N\) is of course null homotopic. The same is true of the similarly defined \(\alpha_M\) and \(\beta_M\). 

Now consider the \((n+m)\)-dimensional Poincar\'e Duality complex \(N \times M\). Still regarding the \(CW\)-complex structures of \(N\) and \(M\) as fixed, our assumptions on \(N\) and \(M\) are sufficient to guarantee local compactness for both, in the sense of \cite{hatcher}, and thus \(N \times M\) is also a \(CW\)-complex (see also \cite{whitehead_combinatorial}). Indeed, \(N\times M\) inherits a natural \(CW\)-complex structure; in particular, its \((n+m-1)\)-skeleton \(\overline{N\times M}\) is realised in our above notation by the homotopy pushout square
\begin{equation}\label{dgm:productpushout1}
    \begin{tikzcd}[row sep=3em, column sep=3em]
        S^{n-1}\times S^{m-1} \arrow[r, "\alpha_N\times f_M"] \arrow[d, "f_N\times\alpha_M"] & N\times\overline{M} \arrow[d] \\
        \overline{N}\times M \arrow[r] & \overline{N\times M}.
    \end{tikzcd}
\end{equation}

This square can be refined further. Recall that the sphere \(S^{m+n-1}\simeq \partial (D^n\times D^m)\) is recovered from the pushout defining the boundary of a product of two discs, where we use \(\iota\) to denote the inclusion of a sphere as the boundary of a disc:
\begin{equation}\label{dgm:spherepushout_boundary}
    \begin{tikzcd}[row sep=3em, column sep=3em]
        S^{n-1}\times S^{m-1} \arrow[r, "\iota\times 1"] \arrow[d, "1\times\iota"] & D^n\times S^{m-1} \arrow[d] \\
        S^{n-1}\times D^{m} \arrow[r] & \partial (D^n\times D^m). 
    \end{tikzcd}
\end{equation} 
Since \((f_N\times\beta_M)\circ(1\times\iota)\simeq f_N\times\alpha_M\), and vice versa for \(\alpha_N\times f_M\), the pushout square (\ref{dgm:productpushout1}) thus factors in the following way, realising the attaching map \(f_{N \times M}\) of the top-cell of the product as an induced map
\begin{equation}\label{dgm:productpushout2}
    \begin{tikzcd}[row sep=3em, column sep=3em]
        S^{n-1}\times S^{m-1} \arrow[r, "\iota\times1"] \arrow[d, "1\times\iota"] & D^n\times S^{m-1} \arrow[d] \arrow[r, "\beta_N\times f_M"] & N\times\overline{M} \arrow[dd] \\
        S^{n-1}\times D^m \arrow[r] \arrow[d, "f_N\times \beta_M"] & S^{n+m-1} \arrow[dr, dashed, "f_{N \times M}"] & \\
        \overline{N}\times M \arrow[rr] && \overline{N \times M}.
    \end{tikzcd}
\end{equation}

\begin{rem}
    Note that Diagram (\ref{dgm:productpushout2}) is equivalent to the homotopy commutative cube
    \begin{equation}\label{dgm:product_cube}
        \begin{tikzcd}[row sep=2em, column sep=2em]
            S^{n-1} \times S^{m-1} \arrow[rr, equals] \arrow[dr, "1\times\iota"] \arrow[dd,  "\iota\times1"] && S^{n-1} \times S^{m-1} \arrow[dd, near end, "\alpha_N\times f_M"] \arrow[dr, "f_N\times\alpha_M"] \\
            & S^{n-1}\times D^m \arrow[rr, crossing over, near start, "f_N\times\beta_M"] && \overline{N}\times M \arrow[dd] \\ 
            D^n\times S^{m-1} \arrow[rr, near end, "\beta_N\times f_M"] \arrow[dr] && N\times\overline{M} \arrow[dr] \\
            & S^{n+m-1} \arrow[rr, dashed, "f_{N\times M}"] \arrow[from=uu, crossing over] && \overline{N\times M}
        \end{tikzcd}
    \end{equation}
    in which both the left- and right-hand vertical faces are homotopy pushouts.
\end{rem}

\begin{exa}\label{exa:sphereprod_skeleton}
    Let \(r>1\) be a (finite) integer, take a set \(\{n_1,\dots, n_r\in\N:n_i>1\text{ for all }i=1,\dots,r \}\) and consider the product of spheres \(P=\prod_{i=1}^r S^{n_i}\). This construction recovers the fact that \(\overline{P}\) has the homotopy type of \(\mathrm{FW}(S^{n_1},\dots, S^{n_r})\), the \textit{fat wedge}, defined generally as
    \[
        \mathrm{FW}(X_1,\dots, X_m)=\Big\lbrace(x_1,\dots,x_m)\in\prod_{i=1}^m X_i:x_i=\ast\text{ for at least one } i=1,\dots,m\Big\rbrace.
    \]
    Note that it is sufficient for our purposes in this article to realise the attaching map of the top-cell via Diagram (\ref{dgm:product_cube}), and to leave the homotopy class of the attaching map ambiguous. Indeed, we make no claims towards a decomposition of what (as in this example) can be a very complicated map. 
\end{exa}

\section{Gyrations}\label{sec:gy}

Let \(N\) be a path-connected \(n\)-dimensional Poincar\'e Duality complex with a single \(n\)-cell and, as in Section \ref{sec:products}, let \(\overline{N}\) be the \((n-1)\)-skeleton of \(N\) and write \(f_N\) for the attaching map of its top-cell. Let \(k\geq 2\) be an integer and take a homotopy class \(\tau\in\pi_{k-1}(\mathrm{SO}(n))\). Then using the standard linear action of \(\mathrm{SO}(n)\) on \(S^{n-1}\) define the map 
\[
    t:S^{n-1}\times S^{k-1}\rightarrow S^{n-1}\times S^{k-1}
\] 
by \(t(a, x)=(\tau(x)\cdot a,x)\). 

\begin{defn}\label{def:gy}
    Let \(k\geq2\) be an integer and let \(N\) be a path-connected \(n\)-dimensional Poincar\'e Duality complex with a single \(n\)-cell. Define the \textit{\(k\)-gyration of \(N\) by \(\tau\)} to be the space defined by the (strict) pushout
    \begin{equation}\label{dgm:gyrationdef}
        \begin{tikzcd}[row sep=3em, column sep=3em]
            S^{n-1}\times S^{k-1} \arrow[r, "1\times \iota"] \arrow[d, "(f_N\times 1)\circ t"] & S^{n-1}\times D^k \arrow[d] \\
            \overline{N}\times S^{k-1} \arrow[r] & \mathcal{G}^k_\tau(N) 
        \end{tikzcd}
    \end{equation} 
    where \(\iota\) is the inclusion of the boundary of the disc. When the context is clear, we will usually just write \textit{gyration} for \(\mathcal{G}^k_\tau(N)\).
\end{defn} 

An important special case is when \(\tau\) is the trivial class, in which case \(t\) is the identity map and the gyration is written as \(\mathcal{G}_{0}^{k}(N)\). We call this the trivial \(k\)-gyration

\begin{rem}
In general, a \(k\)-gyration \(\mathcal{G}^{k}_{\tau}(N)\) is a Poincar\'{e} Duality complex of dimension \(n+k-1\). There is an alternative surgery definition (see for example \cite{huang_inertness24}*{Section 12}) whereby the trivial \(k\)-gyration is expressed as a \((k-1,n)\)-type surgery on \(N\times S^{k-1}\). When \(\tau\) is non-trivial the surgery is twisted, and for this reason we will often refer to the class \(\tau\) as a \textit{twisting} in the context of gyrations. Furthermore, via this surgery definition it follows that if \(N\) has the homotopy type of an oriented manifold then \(\mathcal{G}^{k}_{\tau}(N)\) is an \((n+k-1)\)-manifold with an orientation inherited from that of \(N\).
\end{rem}

As in the Introduction, there is a notion of stability, formulated in the following way.

\begin{defn}
    Let \(N\) be a path-connected \(n\)-dimensional Poincar\'e Duality complex with a single \(n\)-cell. For a given \(k\geq2\), \(M\) is called \textit{\(\mathcal{G}^k\)-stable} if \(\mathcal{G}^{k}_{\tau}(N)\simeq\mathcal{G}^{k}_{\omega}(N)\) for all twistings \(\tau,\omega\in\pi_{k-1}(\mathrm{SO}(n))\). 
\end{defn}
    
When the context is clear this property is called \textit{gyration stability}. Note that \(N\) is \(\mathcal{G}^k\)-stable if and only if for all twistings \(\tau\) there is a homotopy equivalence \(\mathcal{G}^{k}_{\tau}(N)\simeq\mathcal{G}^{k}_0(N)\). 

We need to reframe some aspects slightly for our later work. For a given twisting \(\tau\), take \(x\in S^{k-1}\) and write \(\tau_x:S^{n-1}\rightarrow S^{n-1}\) for the self-map \(a\mapsto \tau(x)\cdot a\). Thus for each \(x\), the composite \((f_N\times1)\circ t\) of (\ref{dgm:gyrationdef}) performs 
\[
    (a,x)\longmapsto(f_N\circ\tau_x(a), x).
\] 

\begin{lem}\label{lem:gystab}
    Let \(N\) be a path-connected \(n\)-dimensional Poincar\'e Duality complex with a single \(n\)-cell and take an integer \(k\geq2\). Then \(N\) is \(\mathcal{G}^k\)-stable if and only if there is a homotopy \(f_N\circ\tau_x\simeq f_N\) for all \(x\in S^{k-1}\) and for all \(\tau\in\pi_{k-1}(\mathrm{SO}(n))\).
\end{lem}

\begin{proof}
    If \(N\) is \(\mathcal{G}^k\)-stable then there is a homotopy equivalence \(\mathcal{G}^{k}_{\tau}(N)\simeq\mathcal{G}^{k}_0(N)\) for all \(\tau\), and note that if \(\tau\) is the trivial class then every \(\tau_x\) is homotopic to the identity map. Thus in the \(\mathcal{G}^k\)-stable case the homotopy pushout (\ref{dgm:gyrationdef}) forces a homotopy \(f_N\circ\tau_x\simeq f_N\) for all \(x\in S^{k-1}\) and for all \(\tau\). This holds since, if there existed a \(\tau\) and \(x\) for which this homotopy did not hold, then \(\mathcal{G}^{k}_{\tau}(N)\not\simeq\mathcal{G}^{k}_0(N)\), contradicting the hypothesis that \(N\) is \(\mathcal{G}^k\)-stable.
    
    Conversely, if for some twisting \(\tau\) we had \(f_N\circ\tau_x\simeq f_N\) for all \(x\in S^{k-1}\), then by construction this implies that \((f_N\times1)\circ t\simeq f_N\times 1\), and so the pushout (\ref{dgm:gyrationdef}) gives a homotopy equivalence \(\mathcal{G}^{k}_{\tau}(N)\simeq\mathcal{G}^{k}_0(N)\). It follows that if this holds for every \(\tau\) then \(N\) is \(\mathcal{G}^k\)-stable.
\end{proof}

\section{The Action of \(\pi_{k-1}(\mathrm{SO}(n+m))\) via the Join}\label{sec:action}

In order to understand gyrations of the product \(N\times M\) we must first understand \(\pi_{k-1}(\mathrm{SO}(n+m))\). To emphasise, \(n\) comes in our context as the dimension of \(N\); we assumed that \(n>2\), and we may consider this as fixed for what follows in this section (without loss of generality). Similarly, one can consider \(m>2\) to be some other fixed integer. When \(k<n\) there is a well known isomorphism 
\[
    \pi_{k-1}(\mathrm{SO}(n+m))\cong\pi_{k-1}(\mathrm{SO}(n))
\]
induced via the canonical map which takes a matrix \(A\in\mathrm{SO}(n)\) and includes into \(SO(n+m)\) by adjoining an identity matrix \(I_m\). Given \(\tau\in\pi_{k-1}(\mathrm{SO}(n+m))\) we will write \(\tau'\) for its image in \(\pi_{k-1}(\mathrm{SO}(n))\) under this isomorphism.

The way in which such an element of \(\pi_{k-1}(\mathrm{SO}(n+m))\) acts on \(S^{n+m-1}\) is understood via the join. Recall that the join \(A\ast B\) is a quotient space of \(A\times B\times I\) where one identifies \((a,b,1)\) to \((a,\ast,1)\), \((a,b,0)\) to \((\ast,b,0)\) and \((\ast,\ast,t)\) to \((\ast,\ast,0)\). Moreover, there is a homotopy equivalence \(\Sigma A\wedge B\simeq A\ast B\). For \(A=S^{n-1}\) and \(B=S^{m-1}\) one has a homotopy pushout
\begin{equation*}\label{dgm:spherepushout_1}
    \begin{tikzcd}[row sep=3em, column sep=3em]
        S^{n-1}\times S^{m-1} \arrow[r, "\iota\times 1"] \arrow[d, "1\times\iota"] & D^n\times S^{m-1} \arrow[d] \\
        S^{n-1}\times D^{m} \arrow[r] & S^{n-1}\ast S^{m-1}. 
    \end{tikzcd}
\end{equation*} 
Note that the above pushout square is equivalent, in the sense of Arkowitz \cite{ark}, to (\ref{dgm:spherepushout_boundary}). When \(k<n\) and given \(\tau\in\pi_{k-1}(\mathrm{SO}(n+m))\), for a fixed \(x\in S^{k-1}\) the action of \(\tau(x)\) on \(S^{n+m-1}\) is thus recovered by acting as \(\tau'(x)\) on \(S^{n-1}\) and acting trivially on \(S^{m-1}\) and then assembling via the join in the following way
\begin{equation}\label{eq:join}
    \tau(x)\cdot(a,b,t)/_\sim \longmapsto (\tau'(x)\cdot a,b,t)/_\sim.
\end{equation}
for a general point \((a,b,t)/_\sim\in S^{n-1}\ast S^{m-1}\). We record this as the following lemma.

\begin{lem}\label{lem:action}
    Let \(k<n\) and take a homotopy class \(\tau \in \pi_{k-1}(\mathrm{SO}(n+m))\). Then for any \(x\in S^{k-1}\) there is a homotopy commutative diagram
    \begin{equation*}
        \begin{tikzcd}[row sep=3em, column sep=3em]
            S^{n-1}\times S^{m-1} \arrow[r, "\iota\times1"] \arrow[d, "1\times \iota"] & D^n\times S^{m-1} \arrow[d] \arrow[r, "1\times1"] & D^n\times S^{m-1} \arrow[dd] \\
            S^{n-1}\times D^m \arrow[r] \arrow[d, "\tau'_x\times1"] & S^{n+m-1} \arrow[dr, dashed, "\tau_x"] & \\
            S^{n-1}\times D^m \arrow[rr] && S^{n+m-1}
        \end{tikzcd}
    \end{equation*}
    in which both the outer and the upper-left squares are homotopy pushouts, inducing the map \(\tau_x\). \qed
\end{lem}

Note that Lemma \ref{lem:action} is implies the existence of the homotopy commutative cube
\begin{equation}\label{dgm:action_cube}
    \begin{tikzcd}[row sep=2em, column sep=2em]
        S^{n-1} \times S^{m-1} \arrow[rr, "\tau'_x\times1"] \arrow[dd, "\iota\times1"] \arrow[dr, "1\times\iota"] && S^{n-1} \times S^{m-1} \arrow[dr, "1\times\iota"] \arrow[dd, near end, "\iota\times1"] \\
        & S^{n-1}\times D^m \arrow[rr, crossing over, near start, "\tau'_x\times1"] && S^{n-1}\times D^m \\ 
        D^n\times S^{m-1} \arrow[rr, equals] \arrow[dr] && D^n\times S^{m-1} \arrow[dr] \\
        & S^{n+m-1} \arrow[rr, dashed, "\tau_x"] \arrow[from=uu, crossing over] && S^{n+m-1}. \arrow[from=uu, crossing over]
    \end{tikzcd}
\end{equation} 

\section{Products with a \(\mathcal{G}^k\)-stable Poincar\'e Duality Complex}\label{sec:mainthm}

Let \(N\) and \(M\) be two Poincar\'e Duality complexes, of dimension \(n\) and \(m\), respectively. Further, assume for some \(k<n\) that \(N\) is \(\mathcal{G}^k\)-stable. Recall that since \(k<n\) every class in \(\tau\in\pi_{k-1}(\mathrm{SO}(n+m))\) corresponds to a \(\tau'\in\pi_{k-1}(\mathrm{SO}(n))\). 

\begin{thm}\label{thm:main}
    Let \(n,m\) and \(k\) be integers with \(n,m>2\), \(k>1\) and \(n>k\), and suppose that \(N\) and \(M\) are two path connected Poincar\'e Duality complexes of dimension \(n\) and \(m\), respectively. Then, if \(N\) is \(\mathcal{G}^k\)-stable, so is \(N \times M\).
\end{thm}

\begin{proof}
    Let \(x\in S^{k-1}\) and \(\tau\in\pi_{k-1}(\mathrm{SO}(n+m))\) be arbitrary, and consider the combination of (\ref{dgm:product_cube}) and (\ref{dgm:action_cube}) as the homotopy commutative diagram
    \begin{equation*}
        \begin{tikzcd}[row sep=2em, column sep=2em]
            S^{n-1} \times S^{m-1} \arrow[rr, "\tau'_x\times1"] \arrow[dd, "\iota\times1"] \arrow[dr, "1\times\iota"] && S^{n-1} \times S^{m-1} \arrow[rr, equals] \arrow[dr, "1\times\iota"] \arrow[dd, near start, "\iota\times1"] && S^{n-1} \times S^{m-1} \arrow[dd, near end, "\alpha_N\times f_M"] \arrow[dr, "f_N\times\alpha_M"] \\
            & S^{n-1}\times D^m \arrow[rr, crossing over, near start, "\tau'_x\times1"] && S^{n-1}\times D^m \arrow[rr, crossing over, near start, "f_N\times\beta_M"] && \overline{N}\times M \arrow[dd] \\ 
            D^n\times S^{m-1} \arrow[rr, equals] \arrow[dr] && D^n\times S^{m-1} \arrow[rr, near end, "\beta_N\times f_M"] \arrow[dr] && N\times\overline{M} \arrow[dr]\\
            & S^{n+m-1} \arrow[rr, dashed, "\tau_x"] \arrow[from=uu, crossing over] && S^{n+m-1} \arrow[rr, dashed, "f_{N\times M}"] \arrow[from=uu, crossing over] && \overline{N\times M}.
        \end{tikzcd}
    \end{equation*}
    By Lemma \ref{lem:gystab}, if \(N\) is \(\mathcal{G}^k\)-stable, then we have \(f_N\circ\tau'_x\simeq f_N\) for any \(x\) and \(\tau'\). In particular, in this case the horizontal composites in the above combined diagram reduce up to homotopy to the cube (\ref{dgm:product_cube}), and consequently to the pushout diagram (\ref{dgm:productpushout2}). 
    
    Thus, on the one hand we have the composite \(f_{N\times M}\circ\tau_x\) from the above, but on the other hand, the assumption of \(\mathcal{G}^k\)-stability for \(N\) implies that this must have the same homotopy class as the induced map from (\ref{dgm:productpushout2}), which is \(f_{N\times M}\). Therefore, since our choice of \(x\) and \(\tau\) was arbitrary, \(f_{N\times M}\circ\tau_x\simeq f_{N\times M}\) holds for all for all \(x\in S^{k-1}\) and all twistings \(\tau\in\pi_{k-1}(\mathrm{SO}(n+m))\), and so \(N\times M\) is \(\mathcal{G}^k\)-stable.
\end{proof}

We finish by applying Theorem \ref{thm:main} to some examples.

\begin{exa}
    In \cite{ChenTher:gy_stab} Theriault and the author studied gyration stability for projective planes; more precisely, \(\mathcal{G}^k\)-stability was considered for \(\H P^2\) in indices \(k<7\) and \(\O P^2\) in indices \(k<15\). We found that \(\H P^2\) is \(\mathcal{G}^k\)-unstable only when \(k=2\), and that \(\O P^2\) is \(\mathcal{G}^k\)-unstable when \(k=2,4\) or \(12\). Despite these individual instances of instability, Theorem \ref{thm:main} implies that their product \(\H P^2\times\O P^2\) is \(\mathcal{G}^k\)-stable for all \(k\) in the range \(2<k<12\). Of particular curiosity is that whilst both \(\H P^2\) and \(\O P^2\) are unstable for \(k=2\), this does not rule out that their product may be \(\mathcal{G}^2\)-stable.
\end{exa}

\begin{exa}
    Recall that spheres are \(\mathcal{G}^k\)-stable for all \(k\). Therefore one has that for a Poincar\'e Duality complex \(M\) the product \(S^n\times M\) will be \(\mathcal{G}^k\)-stable for all \(k<n\). As a particular instance, take the finite product of spheres \(P=\prod_{i=1}^r S^{n_i}\) as in Example \ref{exa:sphereprod_skeleton}. Then, since we can permute terms in the product, Theorem \ref{thm:main} implies that \(P\) is \(\mathcal{G}^k\)-stable for \(k<\mathrm{max}\{n_i:i=1,\dots ,r\}\). This directly extends \cite{chenery:fico}*{Corollary 6.5 (ii)}, which concerned only a product of two spheres, though make no claims with regards to finding the homotopy type of \(\mathcal{G}^k_0(P)\) when \(r>2\).
\end{exa}

        
\bibliographystyle{amsplain}
\bibliography{bib}

\end{document}